\documentclass[12pt]{article}

\usepackage[T1]{fontenc}

\usepackage{tipa}

\usepackage{tikz}
\usepackage{bbm}
\usepackage{booktabs}
\usetikzlibrary{automata}
\usepackage{latexsym}

\usepackage{amsmath}
\usepackage{amsthm}
\usepackage{amssymb}
\usepackage{mathtools}

\makeatletter
\usepackage{setspace}
	\newcommand{\MSonehalfspacing}{
  		\setstretch{1.44}
  			\ifcase \@ptsize \relax
   		 \setstretch {1.448}
  			\or
   		 \setstretch {1.399}
  			\or
    	 \setstretch {1.433}
  			\fi								}				
	\newcommand{\MSdoublespacing}{
 		\setstretch {1.92}
 			\ifcase \@ptsize \relax
    	\setstretch {1.936}
  			\or
    	\setstretch {1.866}
  			\or
    	\setstretch {1.902}
  			\fi								}
	\MSonehalfspacing
\makeatother
\newtheorem{Satz}{Satz}[section]
\newtheorem{definition}[Satz]{Definition}     
\newtheorem{lemma}[Satz]{Lemma}	
\newtheorem{remark}[Satz]{Remark}	
  
\newtheorem{theorem}[Satz]{Theorem}
\newtheorem{proposition}[Satz]{Proposition}                     
\numberwithin{equation}{section}

\begin{document}
\pagestyle{plain}
\title{Constructive Winning Breaker Strategies in the Maker-Breaker $C_k$-Game}
\author{Matthias Sowa and Anand Srivastav \\ %
	Department of Mathematics, Kiel University,\\
	\{sowa,srivastav\}@math.uni-kiel.de
}
\date{}
\maketitle

\maketitle

\begin{abstract}
Maker-Breaker subgraph games are among the most famous combinatorial games. For $n, q \in \mathbb{N}$ and a fixed 
	subgraph $C$ of the complete graph $K_n$, the two players, 
	called Maker and Breaker, alternately claim edges of $K_n$. 
	Maker claims one previously unclaimed edge per round and Breaker may claim up to $q$ edges per round. 
	If Maker is able to claim all edges of a copy of $C$, he wins the game. Otherwise Breaker wins. 
	Bednarska and Łuczak (2000) determined in a landmark work the exact asymptotics of the treshold bias
	as $\Theta(n^{1/m(C)})$ where $m(C)$ is the 2-density of $C$, analysing suitable random strategies. 
	Since then it has been a major open problem to determine the treshhold bias, if it exists, with
	 corresponding optimal strategies,
	leading to sharp constants in the $\Theta$-notion.
	A famous special case is the triangle game ($C = C_3$) studied by Chvatal and Erdös (1978), 
	who showed that Maker wins if $q \le \sqrt{2n}$
	while Breaker wins if $q \ge 2\sqrt{n}$. Glazik and Srivastav (2022) proved with a new potential function
	method that Breaker wins even if 
	$q \ge \sqrt{8/3}\sqrt{n}$, coming quite close to the $\sqrt{2}\sqrt{n}$ lower bound
	($n$ is sufficiently large).
	Joel Spencer (2019) conjectured that this method might be generalizable to arbitrary subgraphs $C$.
	We prove that this conjecture is true, presenting a general winning strategy for Breaker
	if the potential function fullfils certain conditions depending on $C$.
	With this result we give the first constructive (polynomial-time)
	strategies for Breaker in the $k$-cycle Maker-Breaker game 
	for arbitrary, but fixed  $k \geq 4$: Breaker wins if $ q>\sqrt[k-1]{(k-1)\big(\frac{2(k-1)}{k}\big)^{k-2}n^{k-2}}$. 
	By Bednarska and Łuczak (2000) our bound is asymptotically optimal. 
	However, our constants are by magnitudes 
	better than those arising from their random strategies. More recently, 
	Sowa and Srivastav (2025) gave the first constructive Maker
	strategy for $C_4$. Our work may motivate the study of Maker strategies for $C_k, k \ge 5$, narrowing the
	gap towards the Breaker bounds presented.
\end{abstract}

\section{Introduction}
\subsection{The Maker-Breaker $C$-game}
Let $C$ be a fixed graph. The Maker-Breaker $C$-game is played on the edges of $K_n$, the complete graph on $n$ vertices. 
The players, called Maker and Breaker alternately claim edges of $K_n$ until each 
edge is claimed by one of the players. 
In each round of the game Maker claims one edge and Breaker claims up to $q$ edges. 
We call $q$ the bias of the game. If the graph of Maker's edges contains a copy of $C$ at the end of the game, 
Maker wins. Otherwise Breaker wins. This is a game of perfect information with no draw, 
so either Maker or Breaker has a winning strategy.
Natural questions are to find the minimum $q$ as a fuction of $n$ such that Breaker 
has a winning strategy and the maxium $q$ dependig on $n$ such that Maker has a winning strategy.
The treshold bias of the game $q_C$ is attained if minimum and maximum coincide, but the proof of 
its existence
is still a challenging open problem, unsolved for any $C$ containing a cycle.
For a comprehensive  introduction to positional games we refer to the monography of 
Hefetz, Krivelevich, Stojaković and Szabó \cite{Hef}, the early papers of Beck \cite{Beck, Beck1982} and his 
foundational work \cite{Beck-Tic}. 
A recent and interesting variant
are so called phantom Maker-Breaker games introduced 
by Clemens, Hamann, Mikalacki, Mogge and Stojaković \cite{Cle}
with random strategies for both players.
\subsection{Previous work}
Bednarska and Łuczak \cite{Bednarska-Luczak} proved that if $C$ contains three non-isolated vertices, 
then there exist constants $c_1,c_2>0$ such that, for sufficiently large $n$, 
Maker wins if $q\leq c_1 n^{1/m(C)}$ and Breaker wins if $q\geq c_2 n^{1/m(C)}$, where
\[
m(C):= \max \big\{  \frac{\vert E(H)\vert-1}{\vert V(H)\vert-1}: H \text{ is a subgraph of }C,\,\vert V(H)\vert \geq 3 \big\}\,.
\]
They conjectured that $c_1$ and $c_2$ could be chosen arbitrarily close to each other, 
but this conjecture remains open for any $C$ containing a cycle. 
In a Bachelor thesis supervised by the authors Lena C. Wolos \cite{Wolos}
calculated upper bounds for $c_1$ resp. lower bounds for $c_2$. 
The gap is quite large: if $C = C_4$, the 4-cycle, then
$c_1 \le 10^{-6}$ and $c_2 \ge 10^{39}$, and the gap increases for more complicated subgraphs $C$. 

In the case of $C=C_3$, Chvátal and Erdős \cite{Chvatal-Erdos} 
showed that Maker wins if $q \le \sqrt{2n}$ while Breaker wins if $q \ge 2\sqrt{n}$. With a randomized Breaker strategy an improvement of the Breaker bias to $q \ge (2 - 1/24)\sqrt{n}$ was given by Balogh and Samotij \cite{Balogh2011}. Glazik and Srivastav \cite{Glazik} gave a winning strategy for Breaker for a bias $q> \sqrt{8/3}\sqrt{n}$, almost reaching the Maker lower bound  $q \le \sqrt{2}\sqrt{n}$. 
Recently, Sowa and Srivastav \cite{Sowa-Sriv} gave the first constructive Maker strategy for winning the $C_4$-game, if $q \le 0.16 n^{2/3}$. Subsequently, Sowa \cite{Sowa} generalized their approach and established a winning strategy for Maker in the $C_6$-game whenever $q \le 0.007 n^{4/5}$.
We would like to
to emphasize that the graph $C$ is fixed, not depending on $n$. 
So cycles for example have a fixed length $k$. The szenario
changes if $C$ depends on $n$, say $C$ is a Hamiltion cycle, or one wish to guarantee connectivity. 
Here other methods are required. For the Hamilton cycle game we refer to 
Krivelevich \cite{Kri} and for the connectivity game to Gebauer and Szabó \cite{GebSz} resp. 
Hefetz, Mikalački, and Stojaković \cite{Hef1},

\subsection{Our contribution}
During a visit of Kiel University in 2019 Joel Spencer raised the conjecture that the potential function method of
Glazik and Srivastav \cite{Glazik} could be generalized to find new Breaker strategies for Maker-Breaker subraph games.
We show that this in fact is true, by presentig a winning strategy for Breaker for any subgraph $C$, if the potential
function fullfils some essential properties depending on $C$ (Theorem 2.6). Extracting such properties from the approach
of Glazik and Srivastav \cite{Glazik} which is very much focussed on triangles ($C = C_3$) is a major challenge. 
To apply this theorem, we design a potential function for the $C_k$-game and obtain 
the first constructive winning strategies for Breaker in the $k$-cycle game, $k>3$, for a bias
$ q>\sqrt[k-1]{(k-1)\big(\frac{2(k-1)}{k}\big)^{k-2}n^{k-2}}$ (Theorem \ref{Theorem: Breaker wins Ck}). Note that 
the constants in our bias are
significantly better than those known from \cite{Bednarska-Luczak}.
The following table summarizes the state-of-the-art in this context.

\begin{table}[ht]
\centering
\caption{Known upper and lower bounds for the bias}
\label{tab:bounds}

\newcommand{\cell}[2]{%
\parbox[c][3.2em][c]{4cm}{%
\centering
#1\\[0.2em]
{\scriptsize #2}%
}%
}

\begin{tabular}{lcc}
\toprule
$k$ & Lower bound & Upper bound \\
\midrule

\addlinespace[0.5em]
$3$
& \cell{$\sqrt{2}\,n^{1/2}$}{Chvátal-Erdős \cite{Chvatal-Erdos}}
& \cell{$\sqrt{\frac{8}{3}}\,n^{1/2}$}{Glazik,Srivastav \cite{Glazik}} \\
\addlinespace[1.2em]

$4$
& \cell{$ 0.16\,n^{2/3}$}{Sowa,Srivastav \cite{Sowa-Sriv}}
& \cell{$ 1.89\,n^{2/3}$}{Theorem \ref{Theorem: Breaker wins Ck}} \\
\addlinespace[1.2em]

$5$
& \cell{$ \Omega(n^{3/4})$}{Bednarska, Łuczak \cite{Bednarska-Luczak}}
& \cell{$ 2.02\,n^{3/4}$}{Theorem \ref{Theorem: Breaker wins Ck}} \\
\addlinespace[1.2em]

$6$
& \cell{$ 0.007\,n^{4/5}$}{Sowa \cite{Sowa}}
& \cell{$ 2.08\,n^{4/5}$}{Theorem \ref{Theorem: Breaker wins Ck}} \\
\addlinespace[0.5em]

\bottomrule
\end{tabular}
\end{table}

\section{The Potential Strategy Theorem}
The Maker-Breaker $C$-game is easy to understand. However, for the formulation of our
potential function, the strategy based on it, and its analysis, we need a formal framework.
\subsection{Formal Definition of the Game}
\begin{definition}\textbf{(Maker-Breaker Subgraph-Game)}\\
For a set $X$ let $\mathcal{P}(X) := \{ A : A \subseteq X \}$ denote the power set of $X$. For $n \in \mathbb{N}$ let
$V_n := [n]$
be the set of vertices and
$E_n := \{ e \in \mathcal{P}([n]) : |e| = 2 \}$
be the set of edges of the complete graph $K_n := (V_n,E_n)$. We define $\mathcal{E}_n := \mathcal{P}(E_n)$.
\begin{enumerate}
\renewcommand{\labelenumi}{(\roman{enumi})}
\item For $q \in \mathbb{N}$ we call a function $s : \mathcal{E}_n \times \mathcal{E}_n \to \mathcal{E}_n$ a $q$-strategy if $|s(M,B)| \leq q$ and
\[
s(M,B) \subseteq M \cup B \Rightarrow M \cup B = E_n
\]
for all $M,B \in \mathcal{E}_n$.
The set $s(M,B)$ contains the edges that a player claims when playing with strategy $s$, if $M$ is the Maker graph and $B$ is the Breaker graph before his turn. For technical reasons it is allowed that this set contains already claimed edges. In this case, only the unclaimed edges in $s(M,B)$ are claimed.
\item We call
\[
S(q) := \{ s : s \text{ is a } q\text{-strategy} \}
\]
the set of $q$-strategies and
\[
S := \bigcup_{q \in \mathbb{N}} S(q)
\]
the set of strategies.
\item For strategies $m,b \in S$ let
\[
M^0(m,b) := \emptyset, \quad B^0(m,b) := \emptyset.
\]
For all $t \in \mathbb{N}$ with $t \geq 1$ we recursively define
\begin{align*}
M^{t+1}(m,b) &:= M^t(m,b) \cup \big(m(M^t(m,b),B^t(m,b)) \setminus B^t(m,b)\big),\\
B^{t+1}(m,b) &:= B^t(m,b) \cup \big(b(M^{t+1}(m,b),B^t(m,b)) \setminus M^{t+1}(m,b)\big).
\end{align*}
$M^t(m,b)$ is the Maker graph and $B^t(m,b)$ is the Breaker graph after the $t$-th turn, if $m$ is the Maker strategy and $b$ is the Breaker strategy.
\item Further, for all $t \in \mathbb{N}$ and $m,b \in S$ we define
\[
G^t(m,b) := \big(M^t(m,b), B^t(m,b)\big).
\]
\item Since
$G^t(m,b) \neq G^{t+1}(m,b)$ or $M^t(m,b) \cup B^t(m,b) = E_n$ for all $t \in \mathbb{N}$, we have
\[
T := \{ t' \in \mathbb{N} : \forall t \in \mathbb{N} \setminus [t'],\; G^{t'}(m,b) = G^t(m,b) \} \neq \emptyset.
\]
Let
\[
\overline{t}(m,b) := \min(T), \quad G(m,b) := (G^t(m,b))_{t \in [\overline{t}(m,b)]}.
\]
$\overline{t}$ is the time at which the game ends because all edges are claimed by either player.
\item If $M^{\overline{t}(m,b)}(m,b)$ contains a copy of $C$, Maker wins the Maker-Breaker $C$-game, otherwise Breaker wins.
\item If Maker is allowed to claim only a single edge while Breaker may claim at most $q \in \mathbb{N}$ edges per turn, i.e.\ Maker chooses a strategy from $S(1)$ and Breaker from $S(q)$, we call the game the Maker-Breaker $C$-game with bias $q$.
\end{enumerate}
Most of the time we have fixed strategies $m$ and $b$ for Maker and Breaker respectively. In this case we usually omit them and write $M^t$ instead of $M^t(m,b)$, $\overline{t}$ instead of $\overline{t}(m,b)$, and so on.
\end{definition}

Now we introduce balance functions and potential functions.

\subsection{Balance Functions and Potential Functions}
\begin{definition}\textbf{(Obligatory and High-Potential Edges)}\label{obl and hp}
Let $C$ be a subgraph of $K_n$. We write $\mathcal{E} := \mathcal{E}_n$, $V := V_n$ and consider the Maker-Breaker $C$-game.
\begin{enumerate}
\renewcommand{\labelenumi}{(\roman{enumi})}
\item We call
\[
pot:(\mathcal{E},\mathcal{E},V)\to \mathbb{R}, \quad (M,B,v)\mapsto pot_{M,B}(v)
\]
a potential function and may think of $M$ and $B$ as the edge sets claimed by Maker and Breaker respectively.
\item For all $M,B \in \mathcal{E}$ and edges $e=\{v,w\}$ let
\[
pot_{M,B}(e) := pot_{M,B}(v) + pot_{M,B}(w).
\]
\item For all $M,B \in \mathcal{E}$ let
\[
obl(M,B) := \big\{ e\in E_n \setminus (M\cup B)\,:\, M\cup\{e\} \text{ contains a copy of } C \big\}
\]
be the set of obligatory edges.
\item For $X,Y\in \mathcal{E}$ and $Z \in \mathcal{E} \setminus \{\emptyset\}$ let
\[
maximal_{X,Y}(Z) \in \{z\in Z : \forall e\in Z,\; pot_{X,Y}(e) \leq pot_{X,Y}(z)\}
\]
be an edge with maximum potential in $Z$ with respect to $X$ and $Y$.
\item Let
\begin{align*}
\overline{f}(M,B) := \min \big\{\, &
q - |obl(M,B)|,\;
n-1-|M|-|B| - |obl(M,B)|
\big\}.
\end{align*}
For all $i \in [\overline{f}(M,B)]$ we recursively define $X(M,B)^i$ and $hp(M,B)^i$ as follows.
For $i=1$ set
\begin{align*}
X(M,B)^1 &:= B \cup obl(M,B),\\
hp(M,B)^1 &:= maximal_{M,X(M,B)^1}\big(E_n \setminus (M \cup X(M,B)^1)\big).
\end{align*}
So $hp(M,B)^1$ is a (free) edge of $(E_n \setminus (M \cup X(M,B)^1)$ with maximum potential according to (iv).
If $i \geq 2$ and $X(M,B)^{i-1}$ and $hp(M,B)^{i-1}$ were already defined, set
\begin{align*}
X(M,B)^i &:= X(M,B)^{i-1} \cup \{hp(M,B)^{i-1}\},\\
hp(M,B)^i &:= maximal_{M,X(M,B)^i}\big(E_n \setminus (M \cup X(M,B)^i)\big).
\end{align*}
\item Finally,
\[
hp(M,B) := \{hp(M,B)^i : i \in [\overline{f}(M,B)]\}
\]
is the set of high-potential edges.
\end{enumerate}
\end{definition}

\begin{remark}
The set $obl(M,B)$ are the edges that Breaker has to claim immediately because otherwise Maker would win within his next turn. After he has claimed the obligatory edges, the unclaimed edge with the highest potential is $hp(M,B)^1$. If Breaker would also claim this edge and update the potentials, the new unclaimed edge with the highest potential would be $hp(M,B)^2$ and so on.
\end{remark}
\begin{definition}\textbf{(Potential Function Strategy)}
\begin{enumerate}
\renewcommand{\labelenumi}{(\roman{enumi})}
\item Given an arbitrary potential function $pot$, we call
\[
b : \mathcal{E} \times \mathcal{E} \to \mathcal{E}, \quad (M,B) \mapsto obl(M,B) \cup hp(M,B)
\]
the potential function strategy for Breaker with respect to the potential function $pot$.
\item Breaker plays according to this strategy if in each round he first claims every edge in $obl(M,B)$ and thereafter successively claims an edge with highest potential among the unclaimed edges, updating the potentials of all vertices and edges. He proceeds in this way until he has claimed $q$ edges, completing his moves in this round.
\end{enumerate}
\end{definition}
The preceding definitions formally describe a Breaker strategy with respect to a general potential function. We now specify a particular potential function for the Maker-Breaker $C$-game, which will lead to Breaker's win if the potential function satisfies certain properties.
\begin{definition}\textbf{(Balance Function)}\label{def: balance}
Let $G=(V,E)$ be a graph, $\theta,q \in \mathbb{R}$, and $I,J \subseteq \mathbb{R}$ with $0 \in I \cap J$.
\begin{enumerate}
\renewcommand{\labelenumi}{(\roman{enumi})}
\item We call $bal : I \times J \to \mathbb{R}_{>0}$ a balance function if there exists a function $b^* : \mathbb{R} \to \mathbb{R}$ such that for all $m \in I$ and $b \in J$ the following three conditions hold:
\begin{enumerate}
\item $bal(m,\cdot) : J \to \mathbb{R}_{>0}, \; z \mapsto bal(m,z)$ is strictly decreasing,
\item $bal(\cdot,b) : I \to \mathbb{R}_{>0}, \; z \mapsto bal(z,b)$ is strictly increasing,
\item $bal_0 := bal(0,0) = bal\big(m,b^*(m)\big) < 1.$
\end{enumerate}
\item We call
\[
d_{bal} : \mathcal{P}(E) \times \mathcal{P}(E) \times V \to \mathbb{R}, \quad (M,B,v) \mapsto b^*\big(\deg_M(v)\big) - \deg_B(v)
\]
the deficit function of $bal$.
\item For all $q \in \mathbb{R}$ we call the function
\[
pot_{bal,q,\theta} : \mathcal{P}(E) \times \mathcal{P}(E) \times V \to \mathbb{R}, (M,B,v) \mapsto pot_{bal,q,\theta}(M,B,v) 
\]
defined by
\[
pot_{bal,q,\theta}(M,B,v) :=
\begin{cases}
0, & \text{if } \deg_M(v)+\deg_B(v) = |V| - 1,\\[4pt]
(1+\theta)^{d_{bal}(M,B,v)/q}, & \text{otherwise}
\end{cases}
\]
the potential function induced by $(bal,q,\theta)$.
\end{enumerate}
\end{definition}

\begin{definition}\label{Balance-Potential Function}\textbf{(Balance-Potential Function)}
Let $C$ be a fixed graph. Let $\overline{n}\in\mathbb{N}$ and $\theta:\mathbb{N}\to\mathbb{R}$ be a decreasing function with ${\displaystyle\lim_{n\to\infty}\theta(n)=0}$ and $q:\mathbb{N}\to\mathbb{R}$ an arbitrary function. We consider the Maker-Breaker $C$-game with bias $q(n)$ on the complete graph $K_n=(V_n,E_n)$ for $n\in \mathbb{N}$. For all $n\in\mathbb{N}$ let $x_n \in \mathbb{R}$ and $bal_n:[0,x_n]\times \mathbb{R}\to \mathbb{R}_{>0}$ be a balance function such that for all $n>\overline{n}$, $v \in V_n$, any Maker graph $M$, and any Breaker graph $B$ the following three properties hold:
\begin{enumerate}
\renewcommand{\labelenumi}{(\roman{enumi})}
\item $(bal_n)_0 = (bal_{\overline{n}})_0$\\
\item $0<pot_{bal_n,q(n),\theta(n)}(M,B,v)\leq 2n \Rightarrow \deg_M(v) \neq \lceil x_n \rceil -1$\\
\item Suppose Breaker plays according to the potential function strategy with respect to the potential function $pot_{bal_n,q(n),\theta(n)}$ and $M$ is the graph of Maker, $B$ is the graph of Breaker respectively and it holds $\Delta(M) < x_n$, and $M'$ is the graph of Maker directly after his next move and $obl(v) := \{o\in obl(M',B)\vert \, v\notin o\}$ are the obligatory Breaker edges not incident in $v$, then the following conditions are satisfied:
\begin{enumerate}
	\renewcommand{\labelenumi}{(\roman{enumi})}
	\item $(bal_n)_0 \bigg( q(n)-\big\vert obl(M',B) \big\vert \bigg)\geq 1$\\
	\item $(bal_n)_0 \bigg( q(n)-\big\vert obl(v) \big\vert \bigg)\geq d_{bal_n}(M',B,v)-d_{bal_n}(M,B,v)$
\end{enumerate}
\end{enumerate}
Then we call the potential function $pot_{bal_n,q(n),\theta(n)}$ a balance-potential induced by $(bal_n)_{n\in \mathbb{N}}$.
\end{definition}

We are now ready to state the central potential function strategy theorem.
\begin{theorem}\textbf{(Breaker's Win)}\label{Theorem: Breaker's win}
If Breaker plays the $C$-subgraph-game according to the potential function strategy with respect to a balance-potential, then there exists $N \in \mathbb{N}$ such that he wins the game for all $n > N$.
\end{theorem}
The proof of Theorem~\ref{Theorem: Breaker's win} requires considerable conceptual and technical effort. 
For the reader's convenience we structure it in suitable subsections.

\subsection{Properties of the potential during the $C$-game}
We assume that Breaker plays the Maker-Breaker $C$-game according to the potential strategy with respect to a balance-potential $pot_{bal_n,q(n),\theta(n)}$ as defined in Definition~\ref{Balance-Potential Function}. Let $n_1 \in \mathbb{N}$ be a natural number with $\overline{n} < n_1$ and
\begin{equation}
\theta(n_1) < \min\left\{ \frac{1}{3}, \frac{1}{(bal_{n_1})_0} - 1 \right\}.
\end{equation}
Further let $\gamma \in (0,1)$, $\eta \in \big(0, 1 - (bal_{n_1})_0(1+\theta(n_1))\big)$ and $\epsilon \in \left(0,\frac{1}{2}\right)$ with
\begin{equation}\label{eq: bal0<...}
\frac{1-\eta}{(1+\epsilon)\big(1+\theta(n_1)\big)(bal_{n_1})_0} > 1.
\end{equation}
For all $n \in \mathbb{N}$ define
\begin{equation}\label{def: c}
c(n) := \left\lceil
\frac{1-\ln(1-\gamma)}
{\ln(1-\eta)-\ln(1+\epsilon)-\ln\big((1+\theta(n))(bal_n)_0\big)}
\right\rceil.
\end{equation}
Since $c(n)$ is bounded and $\lim_{n\to\infty}\theta(n)=0$, we can find $n \in \mathbb{N}$ with $n_1 < n$ and
\begin{equation}\label{eq: n big}
2\,c(n)\,\theta(n) <
\min\Big\{
\eta\gamma,\;
\eta(1-\gamma)\big(1-(1+\epsilon)^{-1/(bal_{n_1})_0}\big)
\Big\}.
\end{equation}
Since $n$ is fixed for the remainder of this proof we omit the dependence of $\theta$ on $n$ and write $\theta$ instead of $\theta(n)$. Accordingly we write $V,E,\mathcal{E},x,bal,q,c$ instead of $V_n,E_n,\mathcal{E}_n,x_n,bal_n,q_n,c(n)$. Let $\mu := 1+\theta$ and let $b^*$ be as in Definition~\ref{def: balance}(i).
For all $M,B \in \mathcal{P}(E)$ set
\begin{equation}\label{def: d short}
d_{M,B} : V \to \mathbb{R}, \quad v \mapsto d_{bal}(M,B,v),
\end{equation}
\begin{equation}\label{def: pot short}
pot_{M,B} : V \to \mathbb{R}, \quad v \mapsto pot_{bal,q,\theta}(M,B,v),
\end{equation}
\begin{equation}
POT_{M,B} := \sum_{v \in V} pot_{M,B}(v).
\end{equation}
Let $b$ be Breaker's potential function strategy with respect to $pot := pot_{bal,q,\theta}$ and let $m$ be an arbitrary Maker strategy. Since these strategies are fixed for the remainder of the proof we omit the dependence on them and write $\overline{t} := \overline{t}(m,b)$ and for each $t \in [\overline{t}]$ we write $M^t := M^t(m,b)$, $B^t := B^t(m,b)$. Assume $e^t \in E$ is the edge with $\{e^t\} = m(M^t,B^t)$, i.e.\ $e^t$ is the edge chosen by Maker in turn $t$.

In the following definition we split up the change in the overall potential $POT$ during a single turn into several parts.
\begin{definition}
For all $M,B \in \mathcal{E}$ let
\begin{equation} 
M'  := M   \cup m(M,B),
B^c := B   \cup obl(M',B),
B'  := B^c \cup hp(M',B)\,.
\end{equation}
For $u,v,w \in V$ with $\big\{\{u,v\}\big\}=m(M,B)$ we define
\begin{enumerate}
\renewcommand{\labelenumi}{(\roman{enumi})}
\item
\begin{equation} \label{def: Delta^+ (w)}
{\Delta}_{M,B}^+(w) := 
\begin{cases}
\mu^{\frac{d_{M',B}(w)}{q}} - \mu^{\frac{d_{M,B}(w)}{q}} & \text{, if } w\in\{u,v\} \\
0 & \text{, otherwise}
\end{cases}\,,
\end{equation}
${\Delta}_{M,B}^+(w)$ is the share of the change in the potential of $w$ caused by Maker's move.
\item
\begin{equation} \label{def: Delta^heads (w)}
{\Delta}_{M,B}^{heads}(w) := 
\begin{cases}
\mu^{\frac{d_{M',B^c}(w)}{q}} - \mu^{\frac{d_{M',B}(w)}{q}} & \text{, if } w\in\{u,v\} \\
0 & \text{, otherwise}
\end{cases}\,,
\end{equation}
${\Delta}_{M,B}^{heads}(w)$ is the share of the change in the potential of $w$ caused by obligatory edges of Breaker if $w$ is one of the vertices in the edge claimed by Maker.

\item
\begin{equation} \label{def: Delta^tails (w)}
{\Delta}_{M,B}^{tails}(w) := 
\begin{cases}
0 & \text{, if } w\in\{u,v\} \\
\mu^{\frac{d_{M',B^c}(w)}{q}} - \mu^{\frac{d_{M',B}(w)}{q}} & \text{, otherwise}
\end{cases}\,,
\end{equation}
${\Delta}_{M,B}^{tails}(w)$ is the share of the change in the potential of $w$ caused by obligatory edges of Breaker if $w$ is contained in the edge claimed by Maker.

\item
\begin{equation} \label{def: Delta^hp (w)}
{\Delta}_{M,B}^{hp}(w) := 
\mu^{\frac{d_{M',B'}(w)}{q}} - \mu^{\frac{d_{M',B^c}(w)}{q}}\,,
\end{equation}
${\Delta}_{M,B}^{hp}(w)$ is the share of the change in the potential of $w$ caused by high-potential edges of Breaker.

\item
\begin{align}\label{def: Delta^0 (w)}
{\Delta}_{M,B}^{0}(w) := 
pot_{M',B'}(w)-pot_{M,B}(w)-{\Delta}_{M,B}^{+}(w) \notag \\
-{\Delta}_{M,B}^{heads}(w)-{\Delta}_{M,B}^{tails}(w)-{\Delta}_{M,B}^{hp}(w)\,,
\end{align}
${\Delta}_{M,B}^{0}(w)$ is the share of the change in the potential of $w$ caused by setting the potential to zero if all edges incident in $w$ were claimed.

\item
\begin{equation} \label{def: Delta^h+ (w)}
{\Delta}_{M,B}^{h+}(w) := 
\begin{cases}
{\Delta}_{M,B}^{heads}(w) 
& \text{, if } {\Delta}_{M,B}^{+}(w)+{\Delta}_{M,B}^{heads}(w) >0 \\
-{\Delta}_{M,B}^{+}(w) 
& \text{, otherwise}
\end{cases}\,,
\end{equation}

\begin{equation} \label{def: Delta^h- (w)}
{\Delta}_{M,B}^{h-}(w) := 
\begin{cases}
0 \,\,\,\,\,\,\,\,\,\,\,\,  \text{, if } {\Delta}_{M,B}^{+}(w)+{\Delta}_{M,B}^{heads}(w) >0 \\
{\Delta}_{M,B}^{+}(w)+{\Delta}_{M,B}^{heads}(w)\,\,\,\,\,\,\, \text{, otherwise}
\end{cases}\,,
\end{equation}
${\Delta}_{M,B}^{h+}(w)$ and ${\Delta}_{M,B}^{h-}(w)$ are defined to ensure ${\Delta}_{M,B}^{h+}(w)+{\Delta}_{M,B}^{h-}(w)={\Delta}_{M,B}^{heads}(w)$, ${\Delta}_{M,B}^{+}(w)+{\Delta}_{M,B}^{h+}(w)\geq 0$, and ${\Delta}_{M,B}^{+}(w)+{\Delta}_{M,B}^{h-}(w)\leq 0$.
\item
\begin{equation} \label{def: Delta^+}
{\Delta}_{M,B}^+ := \sum_{w\in V}{\Delta}_{M,B}^+(w),\,
{\Delta}_{M,B}^{heads} := \sum_{w\in V}{\Delta}_{M,B}^{heads}(w),\,
\end{equation}
\begin{equation} \label{def: Delta^tails}
{\Delta}_{M,B}^{tails} := \sum_{w\in V}{\Delta}_{M,B}^{tails}(w),\,
{\Delta}_{M,B}^{hp} := \sum_{w\in V}{\Delta}_{M,B}^{hp}(w),\,
\end{equation}
\begin{equation}  \label{def: Delta^0}
{\Delta}_{M,B}^{0} := \sum_{w\in V}{\Delta}_{M,B}^{0}(w),\,
{\Delta}_{M,B}^{h+} := \sum_{w\in V}{\Delta}_{M,B}^{h+}(w),\,
\end{equation}
\begin{equation} \label{def: Delta^h-}
 {\Delta}_{M,B}^{h-} := \sum_{w\in V}{\Delta}_{M,B}^{h-}(w),\,
{\Delta}_{M,B} := {\Delta}_{M,B}^+ + {\Delta}_{M,B}^{h+} + (1-\eta){\Delta}_{M,B}^{hp},\,
\end{equation}
\begin{equation}\label{def: r}
r_{M,B} := {\Delta}_{M,B}^{h-} + {\Delta}_{M,B}^{tails} + \eta\,{\Delta}_{M,B}^{hp} +{\Delta}_{M,B}^{0}\,.
\end{equation}
\end{enumerate}
\end{definition}
\begin{definition}\textbf{(Critical Turn)}\label{Def: critical turn}
We call $t\in [\overline{t}-1]$ critical, if ${\Delta}_{M^t,B^t}>0$.
\end{definition}
\begin{proposition}\label{Prop: critical turns}
Let $M,B \in \mathcal{E}$ be Maker's graph and Breaker's graph respectively. Then the following holds:
\begin{enumerate}
\renewcommand{\labelenumi}{(\roman{enumi})}
\item $POT_{M',B'}-POT_{M,B} = {\Delta}_{M,B}^+ + {\Delta}_{M,B}^{heads} + {\Delta}_{M,B}^{tails} + {\Delta}_{M,B}^{hp}+ {\Delta}_{M,B}^{0}$
\item $ {\Delta}_{M,B}^{heads} =  {\Delta}_{M,B}^{h+} +  {\Delta}_{M,B}^{h-}$
\item $r_{M,B} \leq 0$
\item Every $t\in [\overline{t}-1]$ with $POT_{M^{t+1},B^{t+1}}-POT_{M^t,B^t}>0$ is critical.
\end{enumerate}
\end{proposition}
\begin{proof}
\begin{enumerate}
\renewcommand{\labelenumi}{(\roman{enumi})}
\item follows from (\ref{def: Delta^0 (w)}) by summation over all vertices.
\item For all $w\in V$ by (\ref{def: Delta^h+ (w)}) and (\ref{def: Delta^h- (w)}) we have $${\Delta}_{M,B}^{h+}(w) +  {\Delta}_{M,B}^{h-}(w) =  {\Delta}_{M,B}^{heads}(w)$$ if ${\Delta}_{M,B}^{+}(w) +  {\Delta}_{M,B}^{heads}(w)>0$ and  $${\Delta}_{M,B}^{h+}(w) +  {\Delta}_{M,B}^{h-}(w) =-{\Delta}_{M,B}^{+}(w) + {\Delta}_{M,B}^{+}(w)+ {\Delta}_{M,B}^{heads}(w) =  {\Delta}_{M,B}^{heads}(w)$$ otherwise. (ii) follows by summation over all $w\in V$.
\item Since 
${\Delta}_{M,B}^{h-}$, ${\Delta}_{M,B}^{tails}$, and ${\Delta}_{M,B}^{hp}$ are changes in the potential caused by Breaker's edges, all these values are non-positiv. With ${\Delta}_{M,B}^{0}(v) \in \{0, -(1+\theta)^{{d_{bal}(M',B',v)}/q}\}$ for all $v \in V$, (iii) follows.
\item Let $t\in [\overline{t}-1]$ with $POT_{M^{t+1},B^{t+1}}-POT_{M^t,B^t}>0$. We have
\begin{flalign*}
{\Delta}_{M^t,B^t} &= {\Delta}_{M,B}^+ + {\Delta}_{M,B}^{h+} + (1-\eta){\Delta}_{M,B}^{hp} &(\text{by } (\ref{def: Delta^h-}))\\ 
&= {\Delta}_{M,B}^+ + {\Delta}_{M,B}^{heads} - {\Delta}_{M,B}^{h-} + (1-\eta){\Delta}_{M,B}^{hp} &(\text{by } (ii))\\ 
&= POT_{M^{t+1},B^{t+1}}-POT_{M^t,B^t} -r_{M,B} &(\text{by } (\ref{def: Delta^0 (w)}))\\ 
&>0\,. &(r_{M,B}<0)\\ 
\end{flalign*}
Therefore $t$ is critical.
\end{enumerate}
\end{proof}
In the following we pass through a series of Lemmata, revealing useful information about the change of the potential in single turns of the game.
\begin{lemma}\label{Lemma: pot change by M 1}
Let $M,B \in \mathcal{E}$ be Maker's graph and Breaker's graph respectively with $\deg_{M}(v) < x$ for all $v \in V$. Let $e \in E$ with $\{e\} = m(M,B)$ and let $\{u,v\} \in E$ with
\[
M' := M \cup \{\{u,v\}\}, \qquad B' := B \cup \{\{u,v\}\} \neq B.
\]
Then:
\begin{enumerate}
\renewcommand{\labelenumi}{(\roman{enumi})}
\item $pot_{M',B}(w) \leq \mu \, pot_{M,B}(w)$ for all $w \in V$.
\item $\Delta_{M,B}^+ \leq (\mu - 1)\, pot_{M,B}(e)$.
\item
\[
\sum_{w \in V} \big( \mu^{\frac{d_{M,B'}(w)}{q}} - \mu^{\frac{d_{M,B}(w)}{q}} \big)
\leq (\mu^{-\frac{1}{q}} - 1)\, pot_{M,B}(\{u,v\}).
\]
\end{enumerate}
\end{lemma}
\begin{proof}
\begin{enumerate}
\renewcommand{\labelenumi}{(\roman{enumi})}
\item Let $w \in V$ be arbitrary and let
\[
obl(w) := \{ o \in obl(M',B) \mid w \notin o \}.
\]
If $pot_{M,B}(w)=0$ or $pot_{M',B}(w)=0$, then (i) is trivial. Otherwise, since $bal_0<1$, we obtain $q - |obl(w)| \geq d_{M',B}(w) - d_{M,B}(w)$ and $q \geq 1$ by Definition~\ref{Balance-Potential Function}(iii). It follows:
\begin{flalign*}
&& pot_{M',B}(w) - pot_{M,B}(w)
&= pot_{M,B}(w)\left(\frac{pot_{M',B}(w)}{pot_{M,B}(w)} - 1\right) &&\\
&&&= pot_{M,B}(w)\left(\mu^{\frac{d_{M',B}(w)}{q}} \mu^{-\frac{d_{M,B}(w)}{q}} - 1\right) &&\\
&&&= pot_{M,B}(w)\left(\mu^{\frac{d_{M',B}(w)-d_{M,B}(w)}{q}} - 1\right) &&\\
&&&\leq pot_{M,B}(w)\left(\mu^{\frac{q-|obl(w)|}{q}} - 1\right) &&\\
&&&\leq pot_{M,B}(w)(\mu - 1). &&
\end{flalign*}
\item For $v_1,v_2 \in V$ with $e=\{v_1,v_2\}$, by the same argument as above we have:
\begin{flalign*}
&& \Delta_{M,B}^+
&= \sum_{i \in [2]} \Delta_{M,B}^+(v_i)
= \sum_{i \in [2]}\left(\mu^{\frac{d_{M',B}(v_i)}{q}} - \mu^{\frac{d_{M,B}(v_i)}{q}}\right) &&\\
&&&= \sum_{i \in [2]} \left( \mu^{\frac{d_{M,B}(v_i)}{q}}
\left(\mu^{\frac{d_{M',B}(v_i)-d_{M,B}(v_i)}{q}} - 1\right) \right)&&\\
&&&= \sum_{i \in [2]} \mu^{\frac{d_{M,B}(v_i)}{q}}(\mu - 1)
= \sum_{i \in [2]} pot_{M,B}(v_i)(\mu - 1) &&\\
&&&= pot_{M,B}(e)(\mu - 1). &&
\end{flalign*}
\item For any $w \in \{u,v\}$ we have:
\begin{flalign*}
&& \mu^{\frac{d_{M,B'}(w)}{q}} - \mu^{\frac{d_{M,B}(w)}{q}}
&= \mu^{\frac{d_{M,B}(w)}{q}}\left(\mu^{\frac{d_{M,B'}(w)-d_{M,B}(w)}{q}} - 1\right) &&\\
&&&= \mu^{\frac{d_{M,B}(w)}{q}}(\mu^{-\frac{1}{q}} - 1) &&\\
&&&= pot_{M,B}(w)(\mu^{-\frac{1}{q}} - 1). &&
\end{flalign*}
Since
\[
\mu^{\frac{d_{M,B'}(w)}{q}} - \mu^{\frac{d_{M,B}(w)}{q}} = 0
\]
for $w \notin \{u,v\}$, (iii) follows.
\end{enumerate}
\end{proof}
\begin{lemma}\label{Lemma: pot change by M 2}
Let $M,B \in \mathcal{E}$ be Maker's graph and Breaker's graph respectively with $\Delta(M) < x$. Let $e \in E$ with $\{e\} = m(M,B)$ and $M' := M \cup m(M,B)$. Then
\begin{enumerate}
\renewcommand{\labelenumi}{(\roman{enumi})}
\item
$\Delta_{M,B}^{+}(w) + \Delta_{M,B}^{h+}(w)
\leq \big( \mu^{\frac{bal_0 \overline{f}(M',B)}{q}} - 1 \big)\, pot_{M,B}(w)$ for all $w \in V$.
\item
$\Delta_{M,B}^{+} + \Delta_{M,B}^{h+}
\leq \big( \mu^{\frac{bal_0 \overline{f}(M',B)}{q}} - 1 \big)\, pot_{M,B}(e)$.
\end{enumerate}
\end{lemma}
\begin{proof}
\begin{enumerate}
\renewcommand{\labelenumi}{(\roman{enumi})}
\item Let $w \in V$. If $\Delta_{M,B}^{h+}(w) = -\Delta_{M,B}^{+}(w)$, then (i) trivially holds. Otherwise we have $\Delta_{M,B}^{h+}(w) = \Delta_{M,B}^{heads}(w)$ and $\Delta_{M,B}^{+}(w) + \Delta_{M,B}^{heads}(w) > 0$ by (\ref{def: Delta^h+ (w)}). Therefore $w \in e$ and for $B^c := B \cup obl(M',B)$ and $obl(w) := \{o \in obl(M',B) \mid w \notin o\}$ we have
\begin{flalign*}
&& d_{M',B^c}&(w) - d_{M,B}(w)\\
&&&= d_{M',B^c}(w) - d_{M',B}(w) + d_{M',B}(w) - d_{M,B}(w) \\
&&&\leq bal_0\big(q - |obl(w)|\big) + d_{M',B^c}(w) - d_{M',B}(w)
&\llap{(\text{Def. }\ref{Balance-Potential Function}(iii))}\\
&&&\leq bal_0\big(q - |obl(w)| + d_{M',B^c}(w) - d_{M',B}(w)\big) \\
&&&= bal_0\big(q - |obl(w)| - |\{o \in obl(M',B) \mid w \in o\}|\big) \\
&&&= bal_0\big(q - |obl(M',B)|\big) \\
&&&= bal_0 \,\overline{f}(M',B). &\llap{(Definition \ref{obl and hp}(v))}
\end{flalign*}
Therefore
\begin{flalign*}
\Delta_{M,B}^{+}(w) + \Delta_{M,B}^{h+}(w)
&= \Delta_{M,B}^{+}(w) + \Delta_{M,B}^{heads}(w) \\
&= \mu^{\frac{d_{M',B^c}(w)}{q}} - \mu^{\frac{d_{M,B}(w)}{q}}
&\llap{(\text{by }(\ref{def: Delta^+ (w)}), (\ref{def: Delta^heads (w)}))}\\
&= \big(\mu^{\frac{d_{M',B^c}(w)-d_{M,B}(w)}{q}} - 1\big)\mu^{\frac{d_{M,B}(w)}{q}} \\
&= \big(\mu^{\frac{d_{M',B^c}(w)-d_{M,B}(w)}{q}} - 1\big) pot_{M,B}(w) \\
&\leq \big(\mu^{\frac{bal_0\,\overline{f}(M',B)}{q}} - 1\big) pot_{M,B}(w)
\end{flalign*}
and (i) is proved.
\item From (i) and
\[
\forall v \in V \setminus e:\quad \Delta_{M,B}^{+}(v) = \Delta_{M,B}^{h+}(v) = 0
\]
it follows directly that
\begin{align*}
\Delta_{M,B}^{+} + \Delta_{M,B}^{h+}
&= \sum_{v \in V}\big(\Delta_{M,B}^{+}(v) + \Delta_{M,B}^{h+}(v)\big)\\
&= \sum_{v \in e}\big(\Delta_{M,B}^{+}(v) + \Delta_{M,B}^{h+}(v)\big)\\
&\leq \big(\mu^{\frac{bal_0\overline{f}(M',B)}{q}} - 1\big) pot_{M,B}(e).
\end{align*}
\end{enumerate}
\end{proof}

\begin{lemma}\label{Lemma: pot change in critical turns}
Let $t \in [\overline{t}]$ be critical with $\Delta(M^t) < x$ and $e \in E_n \setminus (M^{t+1} \cup B^{t+1})$. Then
\[
pot_{M^{t+1},B^{t+1}}(e)
< \frac{\mu\, bal_0}{1-\eta}\, pot_{M^t,B^t}(e^t).
\]
\end{lemma}
\begin{proof}
For all $i \in [\overline{f}(M^{t+1},B^t)]$, since $X(M^{t+1},B^t)^i \subseteq B^{t+1}$ and Breaker always chooses an edge of maximal potential, we have
\begin{align*}
pot_{M^{t+1},B^{t+1}}(e)
&\leq pot_{M^{t+1},X(M^{t+1},B^t)^i}(e)\\
&\leq pot_{M^{t+1},X(M^{t+1},B^t)^i}\big(hp(M,B)^i\big).
\end{align*}
With Lemma~\ref{Lemma: pot change by M 1}(iii) it follows that
\begin{align*}
\Delta_{M^t,B^t}^{hp}
&= \sum_{w \in V}\Big(
\mu^{\frac{d_{M^{t+1},B^{t+1}}(w)}{q}}
- \mu^{\frac{d_{M^{t+1},X(M^{t+1},B^t)^{\overline{f}(M^{t+1},B^t)}}(w)}{q}}
\Big)\\
&\quad + \sum_{i \in [\overline{f}(M^{t+1},B^t)-1]}
\Big(
\mu^{\frac{d_{M',X(M^{t+1},B^t)^{i+1}}(w)}{q}}
- \mu^{\frac{d_{M',X(M^{t+1},B^t)^i}(w)}{q}}
\Big)\\
&\leq \sum_{i \in [\overline{f}(M^{t+1},B^t)]}
(\mu^{-\frac{1}{q}} - 1)\,
pot_{M^{t+1},X(M^{t+1},B^t)^i}\big(hp(M,B)^i\big)\\
&\leq \overline{f}(M^{t+1},B^t)(\mu^{-\frac{1}{q}} - 1)\,
pot_{M^{t+1},B^{t+1}}(e).
\end{align*}
Therefore we have
\begin{equation}\label{eq: f_bar}
\overline{f}(M^{t+1},B^t)(1-\mu^{-\frac{1}{q}})
pot_{M^{t+1},B^{t+1}}(e)
\leq -\Delta_{M^t,B^t}^{hp}.
\end{equation}
Note that for all $\alpha \geq 1$ we have
\begin{equation}\label{eq: alpha}
\alpha(1-\mu^{-1/q}) \geq 1 - \mu^{-\alpha/q}.
\end{equation}
Set $\alpha := bal_0 \overline{f}(M^{t+1},B^t)$. Note that by Definition~\ref{obl and hp}(v), Definition~\ref{def: balance}(i)(c), and Definition~\ref{Balance-Potential Function}(iii)(a) we have
\begin{equation}\label{eq: alpha<f_bar}
1 \leq \alpha \leq \overline{f}(M^{t+1},B^t) \leq q.
\end{equation}
Now
\begin{flalign*}
&& (1-\eta)\overline{f}(M^{t+1},B^t)&(1-\mu^{-1/q})pot_{M^{t+1},B^{t+1}}(e)
&&\\
&&&\leq -(1-\eta)\Delta_{M^t,B^t}^{hp}
&\llap{(\ref{eq: f_bar})}\\
&&&< \Delta_{M^t,B^t}^{+} + \Delta_{M^t,B^t}^{h+}
&\llap{($t$ is critical)}\\
&&&\leq (\mu^{\alpha/q}-1)pot_{M^t,B^t}(e^t)
&\llap{(Lemma~\ref{Lemma: pot change by M 2}(ii))}\\
&&&= \mu^{\alpha/q}(1-\mu^{-\alpha/q})pot_{M^t,B^t}(e^t)\\
&&&\leq \mu(1-\mu^{-\alpha/q})pot_{M^t,B^t}(e^t)
&\llap{(\ref{eq: alpha})}\\
&&&\leq \alpha\mu(1-\mu^{-1/q})pot_{M^t,B^t}(e^t).
&\llap{(\ref{eq: alpha<f_bar})}
\end{flalign*}
Therefore
\begin{flalign*}
&& pot_{M^{t+1},B^{t+1}}(e)
&< \frac{\alpha\mu}{(1-\eta)\overline{f}(M^{t+1},B^t)}\, pot_{M^t,B^t}(e^t)\\
&&&= \frac{bal_0\,\mu}{1-\eta}\, pot_{M^t,B^t}(e^t).
&\llap{(Definition of $\alpha$)}
\end{flalign*}
\end{proof}

\subsection{Bounding the overall Potential}
In the following we will show that if a $\hat{t} \in [\overline{t}]$ fulfills
\[
POT_{M^t,B^t} < 2n \quad \text{for all } t < \hat{t},
\]
it follows that $POT_{M^{\hat{t}},B^{\hat{t}}} < 2n$, which implies that the overall potential will never exceed the value $2n$, which is a crucial reason for the success of Breaker's strategy. Because the proof of this lemma is quite intricate, we split it into smaller lemmata.

\begin{definition}\label{Def: t_0}
\begin{enumerate}
\renewcommand{\labelenumi}{(\roman{enumi})}
\item Let $\hat{t} \in [\overline{t}]$ with $POT_{M^t,B^t} < 2n$ for all $t \in [\hat{t}-1]$.
Since $POT_{\emptyset,\emptyset} = n$, we have $\big\{ t \in [\hat{t}] : POT_{M^t,B^t} \leq n \big\} \neq \emptyset$.
Define
\[
t_0 := \max \big\{ t \in [\hat{t}] : POT_{M^t,B^t} \leq n \big\} + 1.
\]
Note that $t_0 - 1$ is a critical turn. Let $u,v \in V$ such that $\{u,v\} = e^{t_0-1}$ is the edge claimed by Maker in turn $t_0-1$ and
\[
pot_{M^{t_0-1},B^{t_0-1}}(v) \leq pot_{M^{t_0-1},B^{t_0-1}}(u).
\]
\item For all $t \in [\overline{t}]$,
\[
K(t) := \{ s \in [t-1] : t_0-1 \leq s \text{ and } s \text{ is critical} \}
\]
is the set of critical turns $s$ with $t_0-1 \leq s < t$. Let $k(t) := |K(t)|$. Let
\begin{align}\label{def: T1}
T_1 := \big\{ t\in \mathbb{N} : t_0 \leq t \leq \bar{t} \,\land\, pot_{M^t,B^t}(u) \leq (1-\gamma)pot_{M^{t_0-1},B^{t_0-1}}(u) \big\}\,.
\end{align}
\item
Let
\begin{align}\label{def: T2}
T_2 := \big\{ t\in \mathbb{N} : t_0 \leq t \leq \bar{t} \,\land \,
\exists w\in V \,\,\exists s\in \mathbb{N}\,\, t_0 \leq s < t\, \nonumber \\ 
\land \,\,
pot_{M^t,B^t}(w) \geq (1+\epsilon)pot_{M^s,B^s}(w) \big\}\,.
\end{align}
\item For all $i \in [2]$ set
\begin{align*}
t_i :=
\begin{cases}
\infty & \text{if } T_i = \emptyset,\\
\min(T_i) & \text{otherwise},
\end{cases}
\end{align*}
$t_* := \min\{t_1,t_2\}$ and
\[
R := \sum_{s=t_0-1}^{t_*-1} r_{M^s,B^s}.
\]
\end{enumerate}
\end{definition}

So $t_0-1$ is the last turn in which the overall potential is not larger than $n$. In the following it will be shown that between turn $t_0-1$ and turn $t_*$ there are at most $c$ critical turns and that in these $c$ critical turns the overall potential can not rise above $2n$, where $c$ is the constant defined in (\ref{def: c}). 

\begin{lemma}\label{Lemma: t0<t<t2}
Let $t \in \mathbb{N}$ with $t_0 \leq t < t_2$. Then it holds $POT_{M^t,B^t} < 2n$.
\end{lemma}
\begin{proof}
Since $t < t_2$, we have $t \notin T_2$, so for all $v \in V$, by (\ref{def: T2}) it follows that
\[
pot_{M^t,B^t}(v) < (1+\epsilon)\,pot_{M^{t_0},B^{t_0}}(v).
\]
Hence
\begin{flalign*}
POT_{M^t,B^t}
&= \sum_{v \in V} pot_{M^t,B^t}(v)\\
&< \sum_{v \in V} (1+\epsilon)\,pot_{M^{t_0},B^{t_0}}(v)\\
&= (1+\epsilon)\,POT_{M^{t_0},B^{t_0}}\\
&\leq (1+\epsilon)\Big(
POT_{M^{t_0-1},B^{t_0-1}}
+ \Delta_{M^{t_0-1},B^{t_0-1}}^{+}
+ \Delta_{M^{t_0-1},B^{t_0-1}}^{h+}
\Big)
&\text{(Proposition \ref{Prop: critical turns}(i))}\\
&\leq (1+\epsilon)\,\mu^{\frac{bal_0\overline{f}(M^{t_0},B^{t_0-1})}{q}}
\,POT_{M^{t_0-1},B^{t_0-1}}
&\text{(Lemma \ref{Lemma: pot change by M 2}(i))}\\
&\leq (1+\epsilon)\,\mu\,POT_{M^{t_0-1},B^{t_0-1}}
& (bal_0<1,\ \overline{f}\leq q)\\
&\leq (1+\epsilon)\,\mu\,n
&\text{(Definition \ref{Def: t_0}(i))}\\
&\leq \tfrac{3}{2}\cdot \tfrac{4}{3}\,n = 2n
&\text{(Definition of $\mu,\epsilon$)}
\end{flalign*}
\end{proof}

The following Lemma is a refinement of the inequality in Lemma \ref{Lemma: pot change in critical turns}.
It shows that the potential of an unclaimed edge is bounded by a term exponentially dependent on the number of critical turns.

\begin{lemma}\label{Lemma: bounded pot}
Let $t \in [\bar{t}]$ with $t_0 \leq t < t_2$ and
$e \in E_n \setminus (M^{t} \cup B^{t})$.
Then we have
\[
pot_{M^t,B^t}(e)
< \bigg( \frac{(1+\epsilon)\,\mu\,bal_0}{1-\eta} \bigg)^{k(t)} \, 2 \, pot_{M^{t_0-1},B^{t_0-1}}(u).
\]
\end{lemma}
\begin{proof}
We prove the statement by induction over $k(t)$.

\medskip
\noindent\textbf{$k(t)=1$:}
\begin{flalign*}
pot_{M^t,B^t}(e)
&< (1+\epsilon)\,pot_{M^{t_0},B^{t_0}}(e)
&& \llap{($t<t_2$, (\ref{def: T2}))}\\
&< (1+\epsilon)\,\frac{\mu\,bal_0}{1-\eta}\,
pot_{M^{t_0-1},B^{t_0-1}}(e^{t_0-1})
&& \llap{(Lemma \ref{Lemma: pot change in critical turns})}\\
&\leq \frac{(1+\epsilon)\,\mu\,bal_0}{1-\eta}\,
2\,pot_{M^{t_0-1},B^{t_0-1}}(u).
\end{flalign*}
\medskip
\noindent\textbf{$k(t)>1$:}
For all $s \in [\bar{t}]$ with $t_0 < s < t_2$, $k(s) < k(t)$ and all
$e' \in E_n \setminus (M^{s} \cup B^{s})$, the induction hypothesis yields
\[
pot_{M^s,B^s}(e')
< \bigg( \frac{(1+\epsilon)\,\mu\,bal_0}{1-\eta} \bigg)^{k(s)} \, 2 \, pot_{M^{t_0-1},B^{t_0-1}}(u).
\]
Let $s := \max\{ s' \in [t] : s' \text{ is critical} \} - 1$.
Then $k(s) = k(t) - 1$, and applying the induction hypothesis to $e' = e^s$ gives
\begin{flalign*}
pot_{M^t,B^t}(e)
&< (1+\epsilon)\,pot_{M^{s+1},B^{s+1}}(e)
&& \llap{($t<t_2$, (\ref{def: T2}))}\\
&< (1+\epsilon)\,\frac{\mu\,bal_0}{1-\eta}\,
pot_{M^{s},B^{s}}(e^s)
&& \llap{(Lemma \ref{Lemma: pot change in critical turns})}\\
&< (1+\epsilon)\,\frac{\mu\,bal_0}{1-\eta}
\bigg( \frac{(1+\epsilon)\,\mu\,bal_0}{1-\eta} \bigg)^{k(s)}
\, 2 \, pot_{M^{t_0-1},B^{t_0-1}}(u)\\
&= \bigg( \frac{(1+\epsilon)\,\mu\,bal_0}{1-\eta} \bigg)^{k(s)+1}
\, 2 \, pot_{M^{t_0-1},B^{t_0-1}}(u).
\end{flalign*}
\end{proof}

The following Lemma shows that between round $t_0-1$ and $t_*$ there are at most $c$ critical turns, where $c$ is the constant defined in (\ref{def: c}). 

\begin{lemma}\label{Lemma: k(t*)<c}
We have $k(t_*) \leq c$.
\end{lemma}
\begin{proof}
Let $t := t_* - 1$. Then $t < t_1$, and by definition of $T_1$ we have
\begin{flalign*}
0
&\leq (1-\gamma)\,pot_{M^{t_0-1},B^{t_0-1}}(u)\\
&< pot_{M^t,B^t}(u)
&& \llap{($t<t_1$, (\ref{def: T1}))}
\end{flalign*}
Therefore, by Definition \ref{def: balance}(iii), there exists
$e \in E_n \setminus (M^t \cup B^t)$ with $u \in e$, and we continue:
\begin{flalign*}
\bigg( \frac{(1+\epsilon)\,\mu\,bal_0}{1-\eta} \bigg)^{k(t)} \, 2 \, pot_{M^{t_0-1},B^{t_0-1}}(u)
&> pot_{M^t,B^t}(e)
&& \llap{(Lemma \ref{Lemma: bounded pot})}\\
&> pot_{M^t,B^t}(u)\\
&> (1-\gamma)\,pot_{M^{t_0-1},B^{t_0-1}}(u)
&& \llap{($t<t_1$, (\ref{def: T1}))}\\
&\geq \bigg( \frac{(1+\epsilon)\,\mu\,bal_0}{1-\eta} \bigg)^c \, 2 \, pot_{M^{t_0-1},B^{t_0-1}}(u)
&& \llap{(\ref{def: c})}
\end{flalign*}
Since $\frac{(1+\epsilon)\,\mu\,bal_0}{1-\eta} < 1$ by (\ref{eq: bal0<...}), it follows that
$k(t) < c$, hence $k(t_*) \leq k(t)+1 \leq c$.
\end{proof}

\begin{lemma}\label{Lemma: Bounded POT}
We have
\begin{enumerate}
\renewcommand{\labelenumi}{(\roman{enumi})}
\item
\[
\sum_{s \in K(t_*)} {\Delta}_{M^{s},B^{s}}^+
\leq 2\,c\,(\mu-1)\,pot_{M^{t_0-1},B^{t_0-1}}(u)
\]
\item
\[
POT_{M^{t_*},B^{t_*}} - POT_{M^{t_0-1},B^{t_0-1}}
\leq 2\,c\,(\mu-1)\,pot_{M^{t_0-1},B^{t_0-1}}(u) + R
\]
\end{enumerate}
\end{lemma}
\begin{proof}
\begin{enumerate}
\renewcommand{\labelenumi}{(\roman{enumi})}
\item
\begin{flalign*}
{\Delta}_{M^{t_0-1},B^{t_0-1}}^+
&\leq (\mu-1)\,pot_{M^{t_0-1},B^{t_0-1}}(e^{t_0-1})
&& \llap{(Lemma \ref{Lemma: pot change by M 1}(ii))}\\
&\leq (\mu-1)\,2\,pot_{M^{t_0-1},B^{t_0-1}}(u)
&& \llap{(Definition of $u$)}
\end{flalign*}

For all $s \in K(t_*)$ with $t_0 \leq s$ we have
\begin{flalign*}
{\Delta}_{M^{s},B^{s}}^+
&\leq (\mu-1)\,pot_{M^{s},B^{s}}(e^{s})
&& \llap{(Lemma \ref{Lemma: pot change by M 1}(ii))}\\
&< (\mu-1)\bigg( \frac{(1+\epsilon)\,\mu\,bal_0}{1-\eta} \bigg)^{k(s)}
\,2\,pot_{M^{t_0-1},B^{t_0-1}}(u)
&& \llap{(Lemma \ref{Lemma: bounded pot})}\\
&\leq (\mu-1)\,2\,pot_{M^{t_0-1},B^{t_0-1}}(u)
&& \llap{(\ref{eq: bal0<...})}
\end{flalign*}
Lemma \ref{Lemma: k(t*)<c} implies (i).
\item
\begin{flalign*}
POT_{M^{t_*},B^{t_*}} - POT&_{M^{t_0-1},B^{t_0-1}}
= \sum_{s=t_0-1}^{t_*-1}
\big(POT_{M^{s+1},B^{s+1}} - POT_{M^{s},B^{s}}\big)\\
&= \sum_{s=t_0-1}^{t_*-1}
\big({\Delta}_{M^{s},B^{s}} + r_{M^{s},B^{s}}\big)
&& \llap{(Proposition \ref{Prop: critical turns}(i))}\\
&= R
+ \sum_{s \in K(t_*)} {\Delta}_{M^{s},B^{s}}
+ \sum_{\substack{s \notin K(t_*)\\ t_0-1 \leq s < t_*}} {\Delta}_{M^{s},B^{s}}
&& \llap{(Definition \ref{Def: t_0}(v))}\\
&\leq R + \sum_{s \in K(t_*)} {\Delta}_{M^{s},B^{s}}
&& \llap{(Definition \ref{Def: critical turn})}\\
&\leq R + \sum_{s \in K(t_*)} {\Delta}_{M^{s},B^{s}}^+\\
&\leq 2\,c\,(\mu-1)\,pot_{M^{t_0-1},B^{t_0-1}}(u) + R
&& \llap{(by (i))}
\end{flalign*}
\end{enumerate}
\end{proof}

In the following two Lemmata it will be shown that if the game continues till turn $t_*$, then the overall potential at turn $t_*$ is at most at the potential of turn $t_0-1$. This means the game will be over before turn $t_*$, because otherwise we have a contradiction to the definition of $t_0$.

\begin{lemma}\label{Lemma: t1=t*<t_bar}
Let $t_1=t_*\leq \overline{t}$. Then it holds 
\[
POT_{M^{t_*},B^{t_*}} \leq POT_{M^{t_0-1},B^{t_0-1}}\,.
\]
\end{lemma}
\begin{proof}
It holds
\begin{flalign*}
&& PO&T_{M^{t_*},B^{t_*}} - POT_{M^{t_0-1},B^{t_0-1}} &&\\
&& &\leq 2\,c\,(\mu-1)\,pot_{M^{t_0-1},B^{t_0-1}}(u) + R
&& \llap{(Lemma \ref{Lemma: Bounded POT}(ii))} \\
&& &\leq \eta\,\gamma\,pot_{M^{t_0-1},B^{t_0-1}}(u) + R
&& \llap{(\text{by }(\ref{eq: n big}))} \\
&& &\leq \eta\, \big( pot_{M^{t_0-1},B^{t_0-1}}(u) - pot_{M^{t_*},B^{t_*}}(u) \big) + R
&& \llap{($t_* = t_1$, (\ref{def: T1}))} \\
&& &= -\eta 
\sum_{s \in K(t_*)} \bigg(
{\Delta}_{M^{s},B^{s}}^+(u) + {\Delta}_{M^{s},B^{s}}^{h+}(u)
&&\\
&& &\qquad + {\Delta}_{M^{s},B^{s}}^{h-}(u) + {\Delta}_{M^{s},B^{s}}^{tails}(u)
&&\\
&& &\qquad + {\Delta}_{M^{s},B^{s}}^{free}(u) + {\Delta}_{M^{s},B^{s}}^{0}(u)
\bigg) + R
&& \llap{(Prop. \ref{Prop: critical turns}(i))} \\
&& &\leq -\eta 
\sum_{s \in K(t_*)} \bigg(
{\Delta}_{M^{s},B^{s}}^{h-}(u) + {\Delta}_{M^{s},B^{s}}^{tails}(u)
&&\\
&& &\qquad + {\Delta}_{M^{s},B^{s}}^{free}(u) + {\Delta}_{M^{s},B^{s}}^{0}(u)
\bigg) + R
&& \llap{$(\Delta^+ + h^+ \geq 0)$} \\
&& &\leq -
\sum_{s \in K(t_*)} \bigg(
{\Delta}_{M^{s},B^{s}}^{h-}(u) + {\Delta}_{M^{s},B^{s}}^{tails}(u)
&&\\
&& &\qquad + \eta\,{\Delta}_{M^{s},B^{s}}^{free}(u) + {\Delta}_{M^{s},B^{s}}^{0}(u)
\bigg) + R
&& \llap{$(\gamma \leq 1)$} \\
&& &\leq -
\sum_{s \in K(t_*)} \big(
r_{M^{s},B^{s}}
\big) + R
&& \llap{(\text{by }(\ref{def: r}))} \\
&& &\leq 0 \,.
&& \llap{(Definition of $R$)}
\end{flalign*}
\end{proof}

\begin{lemma}\label{Lemma: t2=t*<t_bar}
Let $t_1 \neq t_2=t_*\leq \overline{t}$. Than it holds
\[
POT_{M^{t_*},B^{t_*}} \leq POT_{M^{t_0-1},B^{t_0-1}}\,.
\]
\end{lemma}
\begin{proof}
Since $t_*=t_2$, by (\ref{def: T2}), there exists $w \in V$ and $s_0 \in [t_*-1]$ with $t_0 \leq s_0$ and
\begin{equation}\label{eq: pot_s<pot_t*}
(1+\epsilon)\,pot_{M^{s_0},B^{s_0}}(w)
\leq pot_{M^{t_*},B^{t_*}}(w)\,.
\end{equation}
First, we show via induction over $t$ that for all $t\in[t_*]$ with $s_0 \leq t$ the following inequality is true:
\begin{equation}\label{eq: pot_t<pot_s}
pot_{M^{t},B^{t}}(w)
\leq
pot_{M^{s_0},B^{s_0}}(w)
\prod\limits_{s=s_0}^{t}
\mu^{\frac{bal_0\overline{f}(M^s,B^{s-1})}{q}} \,.
\end{equation}
For $t=s_0$ the statement is obviously true. So let $t\in[t_*]$ with $s_0<t$ and let the statement be true for $t-1$. We have by induction hypothesis
\begin{align*}
pot_{M^{t},B^{t}}(w)
&= pot_{M^{t},B^{t}}(w)
- pot_{M^{t-1},B^{t-1}}(w)
+ pot_{M^{t-1},B^{t-1}}(w)\\
&\leq
{\Delta}_{M^{t-1},B^{t-1}}^{+}(w)
+ {\Delta}_{M^{t-1},B^{t-1}}^{h+}(w)\\
&\qquad
+ pot_{M^{t-1},B^{t-1}}(w)
&& \llap{(\text{by }(\ref{def: Delta^h+ (w)})-(\ref{def: Delta^h-}))} \\
&\leq
\bigg(
\mu^{\frac{bal_0\overline{f}(M^t,B^{t-1})}{q}} -1
\bigg)
pot_{M^{t-1},B^{t-1}}(w)\\
&\qquad
+ pot_{M^{t-1},B^{t-1}}(w)
&& \llap{(Lemma \ref{Lemma: pot change by M 2})} \\
&=
pot_{M^{t-1},B^{t-1}}(w)\,
\mu^{\frac{bal_0\overline{f}(M^t,B^{t-1})}{q}}\\
&\leq
\bigg(
pot_{M^{s_0},B^{s_0}}(w)
\prod\limits_{s=s_0}^{t-1}
\mu^{\frac{bal_0\overline{f}(M^s,B^{s-1})}{q}}
\bigg)
\mu^{\frac{bal_0\overline{f}(M^t,B^{t-1})}{q}}\\
&=
pot_{M^{s_0},B^{s_0}}(w)
\prod\limits_{s=s_0}^{t}
\mu^{\frac{bal_0\overline{f}(M^s,B^{s-1})}{q}} \,.
\end{align*}
(\ref{eq: pot_s<pot_t*}) and (\ref{eq: pot_t<pot_s}) give
\begin{align*}
(1+\epsilon)\,pot_{M^{s_0},B^{s_0}}(w)
&\leq pot_{M^{t_*},B^{t_*}}(w)
&& \llap{(\text{by }(\ref{eq: pot_s<pot_t*}))} \\
&\leq
pot_{M^{s_0},B^{s_0}}(w)
\prod\limits_{s=s_0}^{t_*}
\mu^{\frac{bal_0\overline{f}(M^s,B^{s-1})}{q}}
&& \llap{(\text{by }(\ref{eq: pot_t<pot_s}))}
\end{align*}
therefore
\begin{align*}
(1+\epsilon)
&\leq
\prod\limits_{s=s_0}^{t_*}
\mu^{\frac{bal_0\overline{f}(M^s,B^{s-1})}{q}}\\
&=
\mu^{\frac{
bal_0
\sum\limits_{s=s_0}^{t_*}
\overline{f}(M^s,B^{s-1})
}{q}}\,.
\end{align*}
Hence
\begin{equation}\label{def: alpha}
\alpha
:=
\frac{q\,\ln(1+\epsilon)}{bal_0\,\ln(\mu)}
\leq
\sum\limits_{s=s_0}^{t_*}
\overline{f}(M^s,B^{s-1})\,.
\end{equation}
Since $t_*<t_1$ we have $pot_{M^{t_*},B^{t_*}}(u)>0$, which means there is still an unclaimed edge $e$ with $u \in e$. By (\ref{eq: alpha<f_bar}) and $t_*<t_1$, for all $s\in[t_*]$ with $t_0 \leq s$ and $i \in [\overline{f}(M^{s},B^{s-1})]$ we have
\begin{align*}
(1-\gamma)\,pot_{M^{t_0-1},B^{t_0-1}}(u)
&< pot_{M^{s},B^{s}}(u)\\
&< pot_{M^{s},B^{s}}(e)\\
&\leq pot_{M^{s},X(M^{s},B^{s-1})^i}(e)\\
&\leq
pot_{M^{s},X(M^{s},B^{s-1})^i}
\big(hp(M^{s},B^{s-1})^i\big)\,.
\end{align*}
For all $s\in[t_*]$ with $t_0 \leq s$ and $i \in [\overline{f}(M^{s},B^{s-1})]$ it follows
\begin{flalign*}
&& PO&T_{M^{s},X(M^{s},B^{s-1})^{i+1}}
- POT_{M^{s},X(M^{s},B^{s-1})^i}
&&\\
&& &\leq
\sum_{v\in V}
\mu^{\frac{d_{M^{s},X(M^{s},B^{s-1})^{i+1}}(v)}{q}} -
\mu^{\frac{d_{M^{s},X(M^{s},B^{s-1})^i}(v)}{q}}
&&\\
&& &\leq
\big(\mu^{-1/q}-1\big)
pot_{M^{s},X(M^{s},B^{s-1})^i}
\big(hp(M^{s},B^{s-1})^i\big)
&& \llap{(Lemma \ref{Lemma: pot change by M 1}(iii))} \\
&& &\leq
\big(\mu^{-1/q}-1\big)
(1-\gamma)\,
pot_{M^{t_0-1},B^{t_0-1}}(u)\,.
&&
\end{flalign*}
So for all $s\in[t_*]$ with $t_0 \leq s$ we have
\begin{flalign*}
&&
\Delta_{M^{s-1},B^{s-1}}^{hp} &=
\sum\limits_{i=1}^{\overline{f}(M^s,B^{s-1})}
\Big(
POT_{M^{s},X(M^{s},B^{s-1})^{i+1}}
-
POT_{M^{s},X(M^{s},B^{s-1})^i}
\Big)
&&\\
&& &\leq
\sum\limits_{i=1}^{\overline{f}(M^s,B^{s-1})}
\big(\mu^{-1/q}-1\big)
(1-\gamma)\,
pot_{M^{t_0-1},B^{t_0-1}}(u)
&&\\
&& &\leq
\overline{f}(M^s,B^{s-1})
\big(\mu^{-1/q}-1\big)
(1-\gamma)\,
pot_{M^{t_0-1},B^{t_0-1}}(u)\,.
&&
\end{flalign*}
Therefore by (\ref{def: alpha})
\begin{equation}\label{eq: pot(u)>...}
\alpha\,(1-\gamma)\big(\mu^{-1/q}-1\big)
pot_{M^{t_0-1},B^{t_0-1}}(u)
\geq
\sum_{s=s_0}^{t_*}
{\Delta}_{M^{s-1},B^{s-1}}^{hp}\,.
\end{equation}
It follows
\begin{flalign*}
&& PO&T_{M^{t_*},B^{t_*}}
- POT_{M^{t_0-1},B^{t_0-1}}
&&\\
&& &\leq
2\,c\,(\mu-1)\,
pot_{M^{t_0-1},B^{t_0-1}}(u) + R
&& \llap{(Lemma \ref{Lemma: Bounded POT}(ii))} \\
&& &\leq
\eta(1-\gamma)
\big(1-(1+\epsilon)^{-1/bal_0}\big)
pot_{M^{t_0-1},B^{t_0-1}}(u) + R
&& \llap{(\text{by }(\ref{eq: n big}))} \\
&& &\leq
\eta(1-\gamma)
\big(1-\mu^{-\alpha/q}\big)
pot_{M^{t_0-1},B^{t_0-1}}(u) + R
&& \llap{(\text{by }(\ref{def: alpha}))} \\
&& &\leq
\alpha\,\eta(1-\gamma)
\big(1-\mu^{-1/q}\big)
pot_{M^{t_0-1},B^{t_0-1}}(u) + R
&& \llap{(\text{by }(\ref{eq: alpha}))} \\
&& &=
-\eta\,\alpha(1-\gamma)
\big(\mu^{-1/q}-1\big)
pot_{M^{t_0-1},B^{t_0-1}}(u) + R
&&\\
&& &\leq
-\eta
\sum_{s=s_0-1}^{t_*-1}
{\Delta}_{M^{s},B^{s}}^{hp}
+ R
&& \llap{(\text{by }(\ref{eq: pot(u)>...}))} \\
&& &\leq
-
\sum_{s=s_0-1}^{t_*-1}
r_{M^{s},B^{s}}
+ R
&& \llap{(\text{by }(\ref{def: r}))} \\
&& &\leq 0\,.
&& \llap{(Definition of $R$)}
\end{flalign*}
\end{proof}

We can now prove that the overall potential at turn $\hat{t}$ is strictly smaller than $2n$.

\begin{lemma}\label{Lemma: POT<2n}
\[
POT_{M^{\hat{t}},B^{\hat{t}}} < 2n \,.
\]
\end{lemma}
\begin{proof}
By Lemma~\ref{Lemma: t0<t<t2}, it is sufficient to show that $\hat{t} < t_2$. Assume for a moment that $t_2 \leq \hat{t}$. Then
\[
t_* \leq t_2 \leq \hat{t} \leq \bar{t}.
\]
Since $t_* \in \{t_1,t_2\}$, we obtain
\[
POT_{M^{t_*},B^{t_*}}
\leq POT_{M^{t_0-1},B^{t_0-1}}
\]
by Lemma~\ref{Lemma: t1=t*<t_bar} in the case $t_* = t_1$, and by Lemma~\ref{Lemma: t2=t*<t_bar} in the case $t_* = t_2$.
We know $t_0 \leq t_*$ by (\ref{def: T1}) and (\ref{def: T2}), and therefore
\[
POT_{M^{t_0-1},B^{t_0-1}}
\leq n
< POT_{M^{t_*},B^{t_*}}
\]
by Definition~\ref{Def: t_0}. It follows that
\[
n
< POT_{M^{t_*},B^{t_*}}
\leq POT_{M^{t_0-1},B^{t_0-1}}
\leq n,
\]
which is a contradiction.
\end{proof}

With the help of the previous lemma we can now prove that Breaker wins the Maker-Breaker-$C$-game with bias $q(n)$ by playing the $pot$-strategy.
\subsection{Proof of Theorem 2.6 }

\begin{proof}
We know that
\[
POT_{M^{t},B^{t}} < 2n
\]
for all $t \in [\overline{t}]$, because otherwise there would be a first point in time $\hat{t}$ with
\[
POT_{M^{\hat{t}},B^{\hat{t}}} \geq 2n,
\]
which is impossible by Lemma~\ref{Lemma: POT<2n}.
Since
\[
POT_{M^{t},B^{t}}
=
\sum_{v \in V} pot_{M^{t},B^{t}}(v),
\]
we also have
\[
pot_{M^{t},B^{t}}(v) < 2n
\]
for all $t \in [\overline{t}]$ and all $v \in V$. By Definition~\ref{Balance-Potential Function}(ii),
\[
0 < pot_{M,B}(v) \leq 2n
\quad \Rightarrow \quad
\deg_M(v) \neq \lceil x \rceil -1,
\]
which implies that the maximum degree of Maker's graph at the end of the game satisfies
\[
\Delta(M^{\overline{t}}) < x.
\]
Therefore, by Definition~\ref{Balance-Potential Function}(iii)(a), we have
\begin{align*}
1
&\leq bal_0\bigl(q - |obl(M^{t+1},B^t)|\bigr)\\
&\leq q - |obl(M^{t+1},B^t)|.
\end{align*}
Hence, in every turn Breaker can claim all edges which may lead to a win for Maker, and therefore Breaker wins the game.
\end{proof}

\begin{remark}
It is easy to see that the worst case running time for Breaker to win the game is
$O(n^4 \log {} n )$:
In each round we have first to update the Maker resp. Breaker degrees.
This concerns $2(q+1)$ vertices. Then we update the potential
for each of the $q + 1$ edges, which can be done in
constant time for each edge.
After updating the potential we must sort the edges in decending order
of its potential, which takes by standard sorting algorithms
 $O(n^2 \log {} n)$ time. From this list Breaker can choose $q$ edges according
to the potential function strategy. So,
 in total the running time is  $ O(n^2/q) O(q) O(n^2 \log {}n) = O(n^4 \log {}n) $.
\end{remark}

\section{The $k$-Cycle Game}
In this section we apply Theorem \ref{Theorem: Breaker's win} to the Maker-Breaker $k$-cycle game in which Maker tries to claim a cycle of length $k\geq3$.\\
We prove that for $q>\sqrt[k-1]{(k-1)\big(\frac{2(k-1)}{k}\big)^{k-2}n^{k-2}}$ Breaker has a winning strategy for sufficiently large $n$. For $n\in\{3,4\}$, these are exactly the bounds shown in \cite{Glazik} and \cite{Sowa-Master}, respectively. For $n>4$ these are the first explicitly known upper bound constants for the Maker-breaker $k$-cycle game.


\begin{theorem}\label{Theorem: Breaker wins Ck}
Let $k\in \mathbb{N}_{\geq 3}$, $\beta > \frac{2(k-1)}{k}$ and $q(n) := \sqrt[k-1]{(k-1)\beta^{k-2}n^{k-2}}$. There exists $N \in \mathbb{N}$ so that for $n>N$ Breaker has a winning strategy for the Maker-Breaker $C_k$-Game with bias $q=q(n)$.
\end{theorem}
For the proof we define balance functions and show that these balance functions fulfill the assumptions of Theorem \ref{Theorem: Breaker's win}.
Let $k\in \mathbb{N}_{\geq 3}$, $C:= C_k$, $n\in \mathbb{N}$, $\bar{\beta} := \frac{2(k-1)}{k}$, $\beta \in (\bar{\beta},(k-1)\bar{\beta})$, $\delta \in (0,1-\frac{\bar{\beta}}{\beta})$, $q:=  \sqrt[k-1]{(k-1)\beta^{k-2}n^{k-2}}$, $x := \sqrt[k-2]{\frac{q}{k-1}}$, $\Theta := \frac{2\,\beta\,\ln(n)\,k}{\delta\,x}$.
Our balance function $bal$ is defined as follows:
\begin{align}\label{def: bal}
bal \colon [0,x] \times \mathbb{R} &\to \mathbb{R} \notag\\
(m,b) &\mapsto
\frac{n-b}{
qx(1-\delta)\frac{k+(k-2)\delta}{2(k-1)}
+
m\bigl(\tfrac{k-2}{2}\,m\,x^{k-3}-q\bigr)
}
\,,
\end{align}
and with $bal_0 := bal(0,0)$, $b^*$ is the function
\begin{align}\label{def: b*}
b^* \colon \mathbb{R} &\to \mathbb{R} \notag\\
m &\mapsto
n
-
bal_0\biggl(
qx(1-\delta)\frac{k+(k-2)\delta}{2(k-1)}
+
m\bigl(\tfrac{k-2}{2}\,m\,x^{k-3}-q\bigr)
\biggr)
\,.
\end{align}
Further $pot := pot_{bal,q,\Theta} $ is defined as in Definition \ref{def: balance}(iii). 
We assume $n$ is sufficiently large. Note that $bal_0$ is independent on $n$ since by definition of $x$ and $q$ we have
\begin{equation}\label{eq: qx=bn}
qx=\beta n\,.
\end{equation}
We use the notions $d_{M,B}(v)$ and $pot_{M,B}(v)$ for $ d_{bal}(M,B,v)$ and $pot_{bal,q,\theta}(M,B,v)$ as in (\ref{def: d short}) and (\ref{def: pot short}).

\begin{lemma}\label{Lemma: bal0}
We have
\begin{enumerate} 
\renewcommand{\labelenumi}{(\roman{enumi})}
\item $bal_0 = \frac{1}{\beta(1-\delta)\frac{k+(k-2)\delta}{2(k-1)}}$
\item $ {\frac{\bar{\beta}}{\beta}} < bal_0
< {\frac{\bar{\beta}}{\beta(1-\delta)}} 
< 1$ .
\end{enumerate}
\end{lemma}
\begin{proof}
\begin{enumerate}
\renewcommand{\labelenumi}{(\roman{enumi})}
\item
By (\ref{eq: qx=bn}),  we know
\[
bal_0
= bal(0,0)
= \frac{n}{
qx(1-\delta)\frac{k+(k-2)\delta}{2(k-1)}
}
= \frac{1}{
\beta(1-\delta)\frac{k+(k-2)\delta}{2(k-1)}
}
\,.
\]
\item
Since $(k-2)\delta > 0$, we get
\[
bal_0
<
\frac{1}{
\beta(1-\delta)\frac{k}{2(k-1)}
}
=
\frac{2(k-1)}{\beta(1-\delta)k}
=
\frac{\bar{\beta}}{\beta(1-\delta)}
< 1.
\]
On the other hand, since $(1-\delta)\bigl(k+(k-2)\delta\bigr) < k$,
we have
\[
bal_0
=
\frac{2(k-1)n}{
qx(1-\delta)\bigl(k+(k-2)\delta\bigr)
}
>
\frac{2(k-1)n}{kqx}
=
\frac{2(k-1)}{k\beta}
=
\frac{\bar{\beta}}{\beta}
\,.
\]
\end{enumerate}
\end{proof}

\begin{lemma}\label{Lemma: deg(v)<x}
Let $M,B \in \mathfrak{P}(E)$ and $v\in V$. For sufficiently large $n$ we have
$$
0<pot_{M,B}(v)\leq2n \Rightarrow \deg_M(v) \neq \lceil x \rceil -1\,.
$$
\end{lemma}
\begin{proof}
We proof the contraposition of the statement. So, let $0<pot_{M,B}(v)$ and $\deg_M(v) = \lceil x \rceil -1$. We show $pot_{M,B}(v)>2n$ for sufficiently large $n$.
By Definition \ref{def: balance}(iii), $pot_{M,B}(v)= (1+\Theta)^{d_{M,B}(v)/q}$. We proceed to show $d_{M,B}(v)\geq \frac{2n\delta}{k}$, and therefore a simple calculation will complete the proof. Note that by definition of $x$, we have 
$x^{k-3}\frac{k-2}{2}x-q= -q\frac{k}{2(k-1)}$. We bound $ b^*(\deg_{M,B}(v))$ from below:
\begin{flalign*}
&& b^*(&\deg_{M,B}(v))=b^*(\lceil x \rceil -1) &&\\
&& &= n - bal_0 \textstyle{\bigg( qx(1-\delta)\frac{k+(k-2)\delta}{2(k-1)} +(\lceil x \rceil -1) \big( x^{k-3}\frac{k-2}{2}(\lceil x \rceil -1)-q \big)\bigg)}&&\\
&& &> n - bal_0 \textstyle{\bigg( qx(1-\delta)\frac{k+(k-2)\delta}{2(k-1)} +(x-1) \big( x^{k-3}\frac{k-2}{2}x-q \big)\bigg)}&&\\
&& &=n - bal_0 \textstyle{\bigg( qx(1-\delta)\frac{k+(k-2)\delta}{2(k-1)} + (x-1) \big( -q\frac{k}{2(k-1)} \big)\bigg)}&&\\
&& &= n + \textstyle{\frac{bal_0}{k-1} \bigg( qx\delta+q\big(\frac{x(k-2)\delta^2-k}{2}\big)\bigg)}&&\\
&& &> n + \textstyle{\frac{bal_0qx\delta}{k-1}}
&& \llap{($x\delta^2> \frac{k}{k-2}$ for large $n$)}\\
&& &= n + \textstyle{\frac{bal_0\beta n\delta}{k-1}} 
&& \llap{(\text{by }(\ref{eq: qx=bn}))}\\
&& &> n + \textstyle{\frac{2 n\delta}{k}}\,.
&& \llap{(Lemma \ref{Lemma: bal0}(ii))}\\
\end{flalign*}
With Definition \ref{def: balance}(ii)
of the deficit function it follows
\begin{equation}\label{eq: d>...}
d_{M,B}(v)
= b^*\big(\deg_M(v)\big)-\deg_B(v) \geq n + \textstyle{\frac{2 n\delta}{k}} -n	
= \textstyle{\frac{2 n\delta}{k}}\,.
\end{equation}
We also need
\begin{equation}\label{eq: ln<...}
\textstyle\big(\frac{1}{\Theta}+1 \big)^{-1} \,\frac{2\delta n}{kq} 
= \frac{\Theta}{\Theta+1} \frac{2\delta n}{kq}
> \frac{\Theta \delta n}{kq}
= \frac{2 \beta \ln(n) n}{qx}
= 2\, \ln(n)\,.
\end{equation}
From (\ref{def: b*}) and (\ref{eq: qx=bn}) we get
\begin{flalign*}
&& pot_{M,B}(v) &= (1+\Theta)^{d_{M,B}(v)/q} \geq (1+\Theta)^{2n\delta/kq}
&& \llap{(\text{by }(\ref{def: b*}))}\\
&& &= (1+\Theta)^{(\frac{1}{\Theta}+1)(\frac{1}{\Theta}+1)^{-1}2n\delta/kq}&&\\
&& &> e^{2\,\ln(n)}
&& \llap{(\text{by }(\ref{eq: qx=bn}))}\\
&& &= n^2 >2n\,,
\end{flalign*}
for $n$ large enough.
\end{proof}

\begin{lemma}\label{Lemma: bal0>...}
Let $v\in V$, $M,B \in \mathfrak{P}(E)$, $e\in E\setminus (M\cup B)$, $M':= M\cup \{e\}$ with $\Delta(M)<x$ and let $obl(v) := \{o\in obl(M',B)\vert \, v\notin o\}$. Then
\begin{enumerate}
	\renewcommand{\labelenumi}{(\roman{enumi})}
	\item $bal_0 \bigg( q-\big\vert obl(M',B) \big\vert \bigg)\geq 1$\\
	\item $bal_0 \bigg( q-\big\vert obl(v) \big\vert \bigg)\geq d_{M',B}(v)-d_{M,B}(v))$
\end{enumerate}
\end{lemma}
\begin{proof}
Let $u,w\in V$ with $e=\{u,w\}$. Note that the obigatory edges of Breaker $obl(M',B)$ are those that connect the end points of a path in $M'$ of length $k-1$ that contain the last claimed Maker edge $e$. Therefore we can bound $\vert obl(M',B) \vert$ from above:
\begin{equation}\label{eq: obl<q}
\vert obl(M',B) \vert \leq\, (k-1) \Delta(M)^{k-2}\leq (k-1) x^{k-2}=q\,.
\end{equation}
Also we can bound $obl(w)$ from above:
\begin{equation}\label{eq: obl(v)<...}
\vert obl(v) \vert \leq\, (k-2) \deg_{M'}(v)\Delta(M)^{k-3}\,.
\end{equation}
(i) We have
\begin{flalign*}
&&  bal_0 \bigg( q&- \vert obl(M',B) \vert \bigg) \geq bal_0 \bigg( q - (k-1) \Delta(M)^{k-2}  \bigg)  
&& \llap{(by (\ref{eq: obl<q}))}\\
&& &= bal_0 \bigg( (k-1) x^{k-2} - (k-1) \Delta(M)^{k-2}  \bigg) 
&& \llap{(Definition of $x$)}\\
&& &\geq bal_0 \bigg( (k-1) x^{k-2} - (k-1) (x-1)^{k-2} \bigg) &&\\
&& &> bal_0  (k-1)&&\\
&& &> \frac{\bar{\beta}}{\beta}  (k-1)
&& \llap{(Lemma \ref{Lemma: bal0}(ii))}\\
&& &>1\,.
\end{flalign*}
(ii) Let $v \in V$. If $v\notin e$, then trivially
$$ d_{M',B}(v)-d_{M,B}(v) \leq 0 \leq bal_0\bigg( q- \vert obl(M',B) \vert \bigg)\,.$$
Otherwise we have $ deg_{M'}(v) =  deg_{M}(v)+1$ and therefore by definition of $d$ and $b^*$
\begin{flalign*}
&&  d_{M',B}(v)-d_{M,B}(v) &=b^*\big( \deg_{M'}(v) \big)-b^*\big( \deg_{M}(v) \big) &&\\
&& &= \textstyle bal_0\bigg( q- (k-2)\, \deg_M(v)x^{k-3} - \frac{k-2}{2}x^{k-3}\bigg) &&\\
&& &\leq bal_0\bigg( q-(k-2)\, \deg_M(v) \Delta (M)^{k-3} \bigg) &&\\
&& &\leq bal_0\bigg( q- \vert obl(v) \vert \bigg)\,.
&& \llap{(by (\ref{eq: ln<...}))}
\end{flalign*}
\end{proof}

{\em Proof of Theorem \ref{Theorem: Breaker wins Ck}:}
Lemma \ref{Lemma: deg(v)<x} and Lemma \ref{Lemma: bal0>...}  ensure that our balance function fullfills the conditions in Definition \ref{def: balance} . Thus Theorem \ref{Theorem: Breaker's win}  can be applied and Breaker wins by playing acoording to the potential function strategy corresponding to this very balance function.\\
\hspace*\fill $\qed$

\section{Concluding Remarks}
It would be interesting to design concrete potential functions in view of our potential function theorem
leading to winning strategies for Breaker for other subgraphs $C$. We think that $C = K_4$ would be the
next and most interesting candidate. Further, as already said, constructive Maker strategies for $k$-cycles,
but also for the notorious $K_4$ would be very interesting. Finally, we may ask whether deterministic polynomial-time
strategies for both players for the subgraph game can de derived derandomizing the random strategies of
Bednarska and Łuczak \cite{Bednarska-Luczak}. 

\bibliographystyle{plain}
\bibliography{sources}

\end{document}